\documentclass[12pt]{article}

\usepackage[utf8]{inputenc}

\usepackage{amsthm}
\usepackage{amsmath}
\usepackage{amssymb}
\usepackage{graphicx}

\newtheorem{theorem}{Theorem}

\newtheorem{lemma}{Lemma}
\newtheorem{corollary}{Corollary}
\newtheorem{remark}{Remark}

\title{Dirichlet Fractional Laplacian in multi-tubes}

\author{F.L. Bakharev\thanks{St.Petersburg State University, Universitetskaya emb. 7-9, St.Petersburg, 199034, Russia, e-mail: f.bakharev@spbu.ru}
\ and
\setcounter{footnote}{6}
A.I. Nazarov\thanks{St.Petersburg Department of Steklov Mathematical Institute of Russian Academy of Sciences, Fontanka 27, St.Petersburg, 191023, Russia, and St.Petersburg State University, Universitetskaya emb. 7-9, St.Petersburg, 199034, Russia, e-mail: nazarov@pdmi.ras.ru}
}

\begin{document}

\maketitle

\noindent{\bf Abstract.} We describe the spectrum structure for the restricted Dirichlet fractional Laplacian in multi-tubes, i.e. domains with cylindrical outlets to infinity. Some new effects in comparison with the local case are discovered.

\medskip


\medskip

\noindent{\bf AMS classification codes:} Primary: 35R11, Secondary: 81Q10.

\section{Introduction}
The goal of this paper is obtaining a better understanding of spectral properties of some non-local operators in domains with cylindrical outlets to infinity. This study has various motivations.

The standard positive Laplacian $-\Delta$ in a domain $\Omega\subset \mathbb{R}^n$ corresponds, up to a multiplicative constant, to the quantization of the kinetic energy $\frac{p^2}{2m}$ of a free particle with momentum $p$ and mass $m$, confined in $\Omega$. This is because the quantization procedure maps the classical momentum $p$ to the operator $-i\nabla$. The Dirichlet condition in this case means the hard walls of the domain. However, the relativity theory tells that the choice of kinetic energy as above is not appropriate for high energies and for a massive relativistic particle it should be replaced by $\sqrt{p^2+m^2}$. Thus, the corresponding quantum Hamiltonian should be chosen as $\sqrt{-\Delta+m^2}$ (see, e.g., \cite{Herbst77, Nardini86, CaMaSi90} for further details). This gives an inspiration to study fractional powers of the Helmholtz operator, especially their spectral properties. Notice that such powers are \textbf{non-local} operators, which significantly complicates the problem.

We discuss mainly the \textbf{fractional Laplacian} though our results can be transferred to the fractional Helmholtz operator.

As in case of a non-relativistic particle, the important complication to the statement of the problem is brought by the boundary condition. In contrast to the local case, we have a non-unique procedure to impose the Dirichlet condition. The first choice is to take the spectral power $(-\Delta_\Omega)^s$ of the conventional Dirichlet Laplacian in $\Omega$. In this case the analysis of spectrum of such a problem reduces to the analysis of the standard Dirichlet Laplacian. 

The second way is to consider the so-called \textbf{restricted} Dirichlet fractional Laplacian  $\mathcal{A}^\Omega_s
$. It is defined by the quadratic form
$$
a_s^\Omega[u]=(\mathcal{A}^\Omega_s u,u):=\int\limits_{\mathbb{R}^n}|\xi|^{2s} |\mathcal{F}_n u(\xi)|^2\,d\xi, 
$$
where $\mathcal{F}_n$ stands for the $n$-dimensional Fourier transform
$$
\mathcal{F}_n u(\xi)=\frac{1}{(2\pi)^\frac n2}\int\limits_{\mathbb{R}^n}e^{-i\xi\cdot x}u(x)\, dx.
$$
The domain of the quadratic form $a_s^\Omega$ is defined as follows:
$$
\mathop{\rm Dom}\nolimits (a_s^\Omega)=\widetilde{H}^s(\Omega):=
\{u\in H^s(\mathbb{R}^n)\colon \mathop{\rm supp}\nolimits u\subset \overline{\Omega}\},
$$
where $H^s(\mathbb{R}^n)$ is the classical Sobolev--Slobodetskii space (see, e.g., \cite[Subsection 2.3.3]{Triebel})
$$
H^s(\mathbb{R}^n)=\{u\in L_2(\mathbb{R}^n)\colon |\xi|^s \mathcal{F}_n u(\xi)\in L_2(\mathbb{R}^n)\}.
$$

In what follows, we assume $s\in (0,1)$. This case has a strong connection to the theory of stochastic processes. While the Laplacian $\Delta$ in $\mathbb{R}^n$ can be considered as a generator of the standard Brownian semigroup $\exp (t\Delta)$, the fractional Laplacian, or more exactly the operator $-(-\Delta)^s$ for $s\in (0,1)$, stands for the generator of the L\'evi-stable motion semigroup. In both cases restricting to the domain $\Omega$ and posing the Dirichlet conditions means posing the killing or absorbing boundary condition for the original random process (see, e.g., \cite{Nardini86, CaMaSi90, GaSt19}).

The study of spectral problems for the conventional Dirichlet Laplacian in domains with cylindrical outlets to infinity has a long history. Typically the spectra of such problems consist of continuous spectra covering the ray $[\lambda_\dagger,+\infty)$ with some positive threshold $\lambda_\dagger$ and a number of eigenvalues (bound states) below the threshold which may appear because of the geometrical structure of the domain in a finite region (the junction). Usually this takes place if it is possible to inscribe a sufficiently large body into the junction (see, e.g., \cite{AvBeGiMa, NazSA, NaRuUu2013, Pa2017}) or if the cylinder is  bended or broken (see, e.g., \cite{ESS, DuEx1995, GoJa1992}).  
Typically, a finite number of eigenvalues may appear under the threshold of the continuous spectrum. In some special cases it is possible to prove the uniqueness of such an eigenvalue (see, e.g., \cite{NazSA, NaRuUu2013, BaMaNa}). 

For the relativistic case, we know only two recent works \cite{ExHo2021} and \cite{BoBrKrOu-2021} which discuss a similar problem for the Dirac operator $-i\nabla$. However, we stress that in contrast to $\mathcal{A}^\Omega_s$, this operator is local.\medskip

The structure of the paper is the following. In Section \ref{s:prelim} we recall some well-known facts about the Caffarelli--Silvestre extension and prove an important auxiliary lemma. Section \ref{s:straight} is devoted to the spectrum of $\mathcal{A}^\Omega_s$ in a (straight) tube. 

In Section \ref{s:mult} we study the spectral properties of $\mathcal{A}^\Omega_s$ in a {\bf perturbed multi-tube}. This means that outside some compact set $\mathcal{K}$, the domain $\Omega$ coincides with a finite union of non-intersecting congruent semi-tubes, cf. \cite{LiNa2019}. We prove that, like in the local case $s=1$, the essential spectrum coincides with that in one semi-tube. However, in comparison with the local case, this result holds only under the following additional assumption: the axes of semi-tubes are not co-directional.

In Section \ref{s:incr} we study the influence of a local widening of a tube on the spectrum of the operator $\mathcal{A}^\Omega_s$. It is well known (see, e.g., \cite{EK} and references therein) that in the local case, arbitrary such widening produces points of the discrete spectrum under the threshold (in other words, the threshold is a {\bf virtual level} for the Dirichlet Laplacian). This effect obviously holds for the {\bf spectral} fractional Laplacian. The same statement turns out to be true for the restricted fractional Laplacian. A bit unexpectedly, the proof for $n=2$, $s\le\frac 12$ is essentially more complicated than in other cases.

\medskip

We use letter $C$ and $c$ (with or without indices) to denote various positive
constants. To indicate that $C$ depends on some parameters, we list 
them in the parentheses: $C(\dots)$.

\section{Caffarelli--Silvestre extensions}
\label{s:prelim}

The relation between fractional differential operators and generalized harmonic extensions
was discovered more than fifty years ago \cite{MO} and became popular thanks to the celebrated work \cite{CaSi}. Namely, given $u\in \widetilde{H}^s(\Omega)$, the function 
\begin{equation}
\label{CS-ext}
U_{s}(x,y)=
\int\limits_{\mathbb{R}^n} {\cal P}_s(x-\widetilde{x},y)u(\widetilde{x})\,d\widetilde{x}, \qquad x\in \mathbb{R}^n,\ y\in\mathbb{R}_+, 
\end{equation}
with the generalized Poisson kernel
$$
{\cal P}_s(x,y)=\frac{\Gamma(\frac{n+2s}{2})}{\pi^{\frac{n}{2}}\Gamma(s)}\,\frac{y^{2s}}{(|x|^2+y^2)^{\frac n2+s}},
$$
minimizes the weighted Dirichlet integral 
$$
\mathcal{E}_s^\Omega (W)=\int\limits_0^\infty \int\limits_{\mathbb{R}^n} y^{1-2s} |\nabla W (x,y)|^2 \,dxdy
$$
over the set 
$$
\mathcal{W} (u)=\{W=W(x,y)\ \colon\  \mathcal{E}_s^\Omega(W)<\infty, \  W|_{y=0}=u\}
$$
and solves the boundary value problem
$$
-\,{\rm div} (y^{1-2s}\nabla W)=0 \quad \mbox{in} \quad \mathbb{R}^n\times \mathbb{R}_+; \quad W|_{y=0}=u.
$$
Moreover, the following relations hold:
\begin{equation}
\label{qq}
\mathcal{A}_s^\Omega u=-C(s)\! \lim\limits_{y \to 
0^+}y^{1-2s}~\!{\partial_y} U_s(\cdot,y),\quad
a_s^\Omega[u]= C(s) \,\mathcal{E}_s^\Omega(U_{s}), 
\end{equation}
where $C(s)=\frac{4^s \Gamma(s+1)}{2s \Gamma(1-s)}$ (the limit is understood in the sense of functionals on $\widetilde{H}^s(\Omega)$ and pointwise at every point of smoothness of $u$).

The function $U_s$ is usually called the Caffarelli--Silvestre extension of $u$. The set $\mathcal{W}(u)$ is also called the set of admissible extensions of $u$.
\medskip

The following statement will be used in Section \ref{s:incr}.

\begin{lemma}\label{L-2}
Let $n>2-2s$.\footnote{This is a restriction only for $n=1$.} Assume that $\Omega$ is bounded. Then for any function $u\in \widetilde{H}^s(\Omega)$, its Caffarelli--Silvestre extension belongs to $L_2(\mathbb{R}^n\times\mathbb{R}_+)$ with weight $y^{1-2s}$.
\end{lemma}

\begin{proof}
Using formula \eqref{CS-ext} and the Fourier transform in $x$ we obtain
\begin{multline*}
I:=\int\limits_0^\infty \int\limits_{\mathbb{R}^{n}}y^{1-2s}|U_s(x,y)|^2\,dx dy=
\int\limits_0^\infty y^{1-2s}\int\limits_{\mathbb{R}^{n}} |\mathcal{F}_{n}U_s(\xi,y)|^2\, d\xi dy\\
=(2\pi)^\frac n2\int\limits_{\mathbb{R}^{n}}|\mathcal{F}_{n}u(\xi)|^2
\int\limits_0^\infty y^{1-2s} |\mathcal{F}_{n}\mathcal{P}_s(\xi,y)|^2\,dy d\xi.
\end{multline*}
Notice that the function $\mathcal{P}_s$ is spherically symmetric in $x$ and homogeneous: $\mathcal{P}_s(x,y)= y^{-n}\mathcal{P}_s(y^{-1}x,1)$. This implies
$$
\mathcal{F}_n\mathcal{P}_s(\xi,y)=\mathcal{F}_n\mathcal{P}_s(y\xi ,1)=:\widehat{p}_s(y|\xi|).
$$
Therefore, we can change the variable in the last integral and obtain
$$
I=(2\pi)^\frac n2\int\limits_{\mathbb{R}^{n}} |\xi|^{2s-2}  |\mathcal{F}_{n}u(\xi)|^2\, d\xi\int\limits_0^\infty t^{1-2s} |\widehat{p}_s(t)|^2 \,dt.
$$
Since $\mathcal{P}_s$ is smooth in $x$, $\widehat{p}_s$ is rapidly (in fact, exponentially) decaying at infinity, and the second integral evidently converges. 
Since $u$ is compactly supported, its Fourier transform is smooth, and the first integral converges for $2-2s<n$. This concludes the proof.
\end{proof}

\section{Spectral problem in a straight tube}
\label{s:straight}

Let $\omega$ be a bounded domain (connected open set) in $\mathbb{R}^{n-1}$, and let $Q$ be a tube (cylinder)
\begin{equation}\label{cyl}
Q =\omega \times \mathbb{R}=\{x=(x',z)\colon x'\in\omega, \, z\in\mathbb{R}\}.
\end{equation}

Recall that the space $\widetilde{H}^s(\omega)$ is compactly embedded into $L_2(\omega)$ and thus the spectrum of the operator $\mathcal{A}_s^\omega$ is purely discrete and consists of a sequence of eigenvalues
$$
0<\lambda_1(\mathcal{A}_s^\omega)<\lambda_2(\mathcal{A}_s^\omega)\le \lambda_3(\mathcal{A}_s^\omega)\le\ldots\le \lambda_k(\mathcal{A}_s^\omega) \le\ldots \to +\infty.
$$
The corresponding sequence of eigenfunctions $\varphi_k(\mathcal{A}_s^\omega)$ can be chosen orthonormal in $L_2(\omega)$. 

The following assertion is more or less standard. We provide its proof for completeness.

\begin{lemma}
The first eigenvalue $\lambda_1(\mathcal{A}_s^\omega)$ (in what follows we denote it by $\Lambda_s$) is simple and the corresponding eigenfunction $\varphi_1(\mathcal{A}_s^\omega)$ can be chosen positive in $\omega$.    
\end{lemma}

\begin{proof}
By \cite[Theorem 3]{MN}, for any $u\in\widetilde{H}^s(\omega)$ we have $|u|\in\widetilde{H}^s(\omega)$, and the inequality $a_s^\omega[|u|]\le a_s^\omega[u]$ holds. Therefore, without loss of generality we can assume $\varphi_1(\mathcal{A}_s^\omega)$ non-negative. Then the strong maximum principle \cite[Theorem 2.5]{IaMoSq2015} (see also \cite{MN2019}) shows that $\varphi_1(\mathcal{A}_s^\omega)>0$ in $\omega$.
Finally, if $\Lambda_s$ was multiple eigenvalue, we could find a sign-changing eigenfunction, a contradiction. 
\end{proof}

The max-min principle (see, e.g., \cite[\S 10.2]{BS}) easily implies that the eigenvalues of the operator $\mathcal{A}_s^\omega$ decrease when the domain $\omega$ expands. 

\begin{remark}
\label{rem1}
The inequality between restricted and spectral fractional Laplacians (\cite[Theorem 2]{MuNa2014}, see also the survey \cite{Na2021}) implies that 
$$
\lambda_k(\mathcal{A}_s^\omega)<\big(\lambda_k(-\Delta_\omega)\big)^s,\qquad k\in\mathbb{N}. 
$$
Exact values of $\lambda_k(-\Delta_\omega)$ are well known for several domains. For $\lambda_k(\mathcal{A}_s^\omega)$, up to our knowledge, no exact values are known, and sufficiently sharp estimates are obtained only in the ball, see \cite{DyKuKw2017} and references therein. We also mention the paper \cite{Kw2012}, where, besides two-sided estimates for $\lambda_k(\mathcal{A}_s^I)$ on the interval $I=(-1,1)$, the two-term asymptotics was derived:
$$
\lambda_k(\mathcal{A}_s^I)=\Big(\frac {k\pi}2-\frac {(1-s)\pi}4\Big)^{2s}+O\Big(\frac 1k\Big),\qquad k\to\infty
$$
(recall that $\lambda_k(-\Delta_I)\equiv \big(\frac {k\pi}2\big)^2$).
\end{remark}

In this section we relate the spectra of $\mathcal{A}_s^Q$ and $\mathcal{A}_s^\omega$.

\begin{theorem}
\label{T:tube}
The spectrum of $\mathcal{A}_s^Q$ coincides with the ray 
\begin{equation}\label{ess}
\sigma(\mathcal{A}_{s}^Q)=\sigma_{ess}(\mathcal{A}_{s}^Q)=[\Lambda_s, +\infty),
\end{equation}
where $\Lambda_s$ is the smallest eigenvalue of $\mathcal{A}_s^\omega$.
\end{theorem}

\begin{proof}
First of all, we recall that, for any semi-bounded self-adjoint operator, the minimum of its spectrum coincides with the minimum of the corresponding Rayleigh quotient.
In particular,
$$
\inf_{v\in \widetilde{H}^s(\omega)} \frac{a^\omega_s[v]}{\|v;L_2(\omega)\|^2}=\Lambda_s\,.
$$

For any $u\in \widetilde{H}^s(Q)$ and $z\in \mathbb{R}$ let us define $u_z\in \widetilde{H}^s(\omega)$ by the formula
$$
u_z(x')=u(x',z),
$$
and denote $U_{z,s}$ its Caffarelli--Silvestre extension. Then we have
\begin{multline*}
a^Q_s[u]=C(s) \int\limits_0^\infty \int\limits_{\mathbb{R}^{n-1}} \int\limits_\mathbb{R} y^{1-2s}\big(|\nabla' U_{s}(x',z, y)|^2+|\partial_z U_{s}(x',z,y)|^2\big) \,dzdx'dy
\\
\geq
C(s) \int\limits_\mathbb{R} \int\limits_0^\infty \int\limits_{\mathbb{R}^{n-1}}  y^{1-2s}|\nabla' U_{z,s}(x',y)|^2 \,dx'dydz 
\\
\geq \int\limits_\mathbb{R} \Lambda_s \|u_z;L_2(\omega)\|^2 \,dz = \Lambda_s\|u;L_2(Q)\|^2,
\end{multline*}
where $\nabla'$ is the gradient with respect to $(x',y)$.
Thus, 
$$
\inf \sigma(\mathcal{A}_s^Q)=\inf_{u\in \widetilde{H}^s(Q)} \frac{a^Q_s[u]}{\|u;L_2(Q)\|^2}\geq \Lambda_s.
$$

To prove \eqref{ess}, we introduce the Dirichlet fractional Helmholtz operator in a domain $\Omega\subset \mathbb{R}^n$
$$
\mathcal{A}_{s,\kappa}^\Omega=(-\Delta_\Omega+\kappa^2)^s.
$$
It is defined by its quadratic form
$$
a_{s,\kappa}^\Omega[u]=(\mathcal{A}_{s,\kappa}^\Omega u,u):=\int\limits_{\mathbb{R}^n}(|\xi|^2+\kappa^2)^{s} |\mathcal{F}_n u(\xi)|^2\,d\xi, \quad u\in \widetilde{H}^s(\Omega).
$$
In the case of a bounded domain $\omega$ we denote by $\Lambda_{s,\kappa}$ the first eigenvalue of the operator $\mathcal{A}_{s,\kappa}^\omega$ and by $\varphi_{s,\kappa}$ the corresponding eigenfunction, which can be chosen positive and normalized in $L_2(\omega)$. 

Obviously, for any $\Omega\subset\mathbb{R}^n$
$$
a_{s,\kappa_1}^\Omega \le a_{s,\kappa_2}^\Omega \quad \mbox{for} \quad \kappa_1\le \kappa_2.
$$
Therefore, the function $f(\kappa)=\Lambda_{s,\kappa}$ is increasing. Moreover, it is continuous, $f(0)=\Lambda_s$, and $f(\kappa)\to+\infty$ as $\kappa\to+\infty$, so 
$$
\{\Lambda_{s,\kappa} \colon \kappa\in [0,+\infty)\}=[\Lambda_s,+\infty).
$$

First, we give an informal explanation of \eqref{ess}.
We claim that the function $\varphi_{s,\kappa}(x')e^{i\kappa z}$ is an ``eigenfunction of continuous spectrum'' for $\mathcal{A}_{s}^Q$ corresponding to the ``eigenvalue'' $\Lambda_{s,\kappa}$. Indeed, we have
$$
\mathcal{F}_n\big[\varphi_{s,\kappa}(x')e^{i\kappa z}\big](\xi)=\mathcal{F}_{n-1}[\varphi_{s,\kappa}](\xi')\delta(\zeta-\kappa),
$$
where $\xi=(\xi',\zeta)$ is the dual variable to $x=(x',z)$, and thus
\begin{multline*}
\mathcal{F}_n\big[\mathcal{A}_{s}^Q \varphi_{s,\kappa}(x')e^{i\kappa z}\big](\xi)=|\xi|^{2s}\mathcal{F}_{n-1}[\varphi_{s,\kappa}](\xi')\delta(\zeta-\kappa)\\
=(|\xi'|^2+\kappa^2)^s\mathcal{F}_{n-1}[\varphi_{s,\kappa}](\xi')\delta(\zeta-\kappa)=\mathcal{F}_{n-1}\big[\mathcal{A}_{s,\kappa}^\omega \varphi_{s,\kappa}\big](\xi')\delta(\zeta-\kappa)\\
=\Lambda_{s,\kappa}\mathcal{F}_{n-1}[\varphi_{s,\kappa}](\xi')\delta(\zeta-\kappa)=\Lambda_{s,\kappa}\mathcal{F}_n\big[\varphi_{s,\kappa}(x')e^{i\kappa z}\big](\xi',\zeta),
\end{multline*}
and the claim follows.

To be more formal, we construct for any $\Lambda_{s,\kappa}$ with $\kappa\geq 0$ a Weyl sequence for the operator $\mathcal{A}_s^Q$. We put
$$
v_m(x)=\varphi_{s,\kappa}(x') \chi_m(z),\quad m\in\mathbb{N},
$$
where
$$
\chi_m(z)=e^{i\kappa z} \chi\left( \frac{z -2m^2}{m}\right),
$$
and $\chi$ is a smooth cutoff function such that  $\chi(z)=1$ for $|z|\le 1$ and $\chi(z)=0$ for $|z|\geq 2$.
One can easily check that $\chi_{m_1}(z)\chi_{m_2}(z)\equiv 0$ if $m_1\ne m_2$, so it is enough to prove that 
$$
\frac{\|\mathcal{A}_s^Q v_m-\Lambda_{s,\kappa} v_m;L_2(Q)\|}{\|v_m;L_2(Q)\|}\to 0 \quad \mbox{as} \quad m\to+\infty.
$$

We have
\begin{multline*}
\mathcal{F}_n\big[\mathcal{A}_s^Q v_m-\Lambda_{s,\kappa} v_m\big](\xi)\\
=\mathcal{F}_n\big[\mathcal{A}_s^Q v_m\big](\xi)- \mathcal{F}_{n-1}\big[\mathcal{A}_s^\omega \varphi_{s,\kappa}\big] (\xi')\mathcal{F}_1[\chi_m] (\zeta)\\
=\big(|\xi|^{2s}-(|\xi'|^2+\kappa^2)^s\big)\, \mathcal{F}_{n-1}[\varphi_{s,\kappa}] (\xi')\mathcal{F}_1[\chi_m] (\zeta),
\end{multline*}
so, by the Parseval theorem we obtain that
\begin{multline*}
\|\mathcal{A}_s^Q v_m-\Lambda_{s,\kappa} v_m;L_2(Q)\|^2 \\
=
\int\limits_{\mathbb{R}^n}\big((|\xi'|^2+|\zeta|^2)^{s}-(|\xi'|^2+\kappa^2)^{s}\big)^2\, |\mathcal{F}_{n-1} [\varphi_{s,\kappa}](\xi')|^2 |\mathcal{F}_1 [\chi_m](\zeta)|^2\, d\xi'd\zeta.
\end{multline*}
We use the relation $|\mathcal{F}_1 [\chi_m](\zeta)|=m\,|\mathcal{F}_1 [\chi](m(\zeta-\kappa))|$, change the variable and arrive at
\begin{multline*}
\|\mathcal{A}_s^Q v_m-\Lambda_{s,\kappa} v_m;L_2(Q)\|^2\\
=m\int\limits_{\mathbb{R}^n}\big((|\xi'|^2+\big|\frac{\tau}{m}+\kappa\big|^2)^{s}-(|\xi'|^2+\kappa^2)^{s}\big)^2\, |\mathcal{F}_{n-1} [\varphi_{s,\kappa}](\xi')|^2 |\mathcal{F}_1 [\chi](\tau)|^2 \, d\xi' d\tau \\
\le m\int\limits_{\mathbb{R}^n}\big(\big|\frac{\tau}{m}+\kappa\big|^{2s}-\kappa^{2s}\big)^2 \, |\mathcal{F}_{n-1} [\varphi_{s,\kappa}](\xi')|^2 |\mathcal{F}_1 [\chi](\tau)|^2 \, d\xi'd\tau\\
\le m\int\limits_{\mathbb{R}}\Big( \big(\frac{\tau}{m}\big)^2+2\kappa|\frac{\tau}{m}\big|\Big)^{2s} |\mathcal{F}_1 \chi(\tau)|^2\,d\tau \int\limits_{\mathbb{R}^{n-1}} |\mathcal{F}_{n-1} \varphi_{s,\kappa}(\xi')|^2 \, d\xi' \le\frac{C_1(s,\kappa)}{m^{2s-1}}.
\end{multline*}
Since $\varphi_{s,\kappa}$ are normalized, we get
$$
\|v_m;L_2(Q)\|^2= \int\limits_\mathbb{R} \left|\chi\left(\frac{z-2m^2}{m}\right)\right|^2\, dz = m\|\chi; L_2(\mathbb{R})\|^2=C_2 m,
$$
and finally
$$
\frac{\|\mathcal{A}_s^Q v_m-\Lambda_{s,\kappa} v_m;L_2(Q)\|}{\|v_m;L_2(Q)\|}\le \frac{C(s,\kappa)}{m^s}\to 0, \quad m\to+\infty,
$$
as desired.
\end{proof}

\begin{corollary}\label{cor1}
Let ${\cal Q} =\omega \times \mathbb{R}_+$ be a semi-tube. Then the spectrum of $\mathcal{A}_s^{\cal Q}$ coincides with the ray $[\Lambda_s, +\infty)$, where $\Lambda_s$ is the smallest eigenvalue of $\mathcal{A}_s^\omega$.
\end{corollary}

Indeed, Theorem \ref{T:tube} and monotonicity of the spectra on domain imply the relation $\sigma(\mathcal{A}_{s}^{\cal Q})\subset[\Lambda_s, +\infty)$, whereas the relation $[\Lambda_s, +\infty)\subset\sigma_{ess}(\mathcal{A}_{s}^{\cal Q})$ holds due to the same Weyl sequence.

\section{Problem in a perturbed multi-tube}
\label{s:mult}

Let $\Omega$ be a {\it perturbed multi-tube}, that is, outside of some compact set $\mathcal{K}$, $\Omega$ coincides with a finite union of non-intersecting semi-tubes ${\cal Q}_j$, $j=1,\dots,N$. We assume that 
\begin{itemize}
    \item all ${\cal Q}_j$ are congruent to ${\cal Q}=\omega \times \mathbb{R}_+$ (recall that $\omega$ is connected);
    \item the axes of ${\cal Q}_j$, $j=1,\dots,N$, are not co-directional, see Fig. \ref{fig:01}.
\end{itemize}

\begin{figure}[ht]
    \centering
    \includegraphics[width=0.6\textwidth]{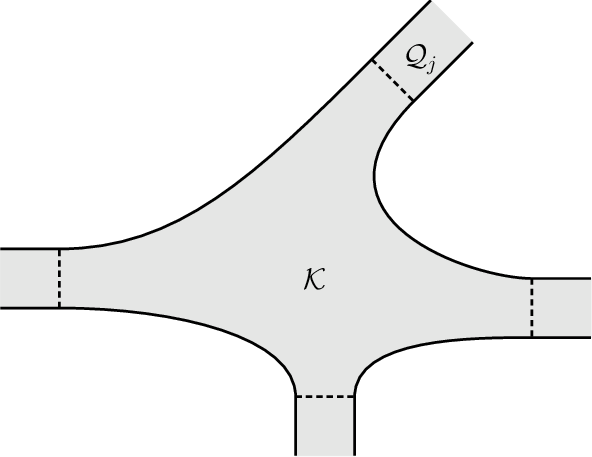}
    \caption{Perturbed multi-tube}
    \label{fig:01}
\end{figure}

\begin{remark}
We stress that the latter assumption is not needed in the local case $s=1$.
\end{remark}

First, we prove an auxiliary statement.

\begin{theorem}
\label{T:eps}
For any $R>0$ and for arbitrary $u\in \widetilde{H}^s(\Omega)$, the following inequality holds:
\begin{equation}
\label{almost}
a^{\Omega}_s[u]\ge \big(\Lambda_s-CR^{-2s}\big) \,\|u; L_2(\Omega)\|^2 - C \|u; L_2(
\mathbb{B}_{R})\|^2,
\end{equation}
where $\mathbb{B}_{R}=\{x\in \mathbb{R}^n \colon |x|<R\}$ is the ball, and $C$ does not depend on $u$ and $R$.
\end{theorem}

\begin{proof}
Here we partly follow the line of the proof of \cite[Lemma 1]{BaNa20} (see also \cite{K}) but essentially modify it for the nonlocal case, cf. \cite[Lemma 3.1]{NS}.
For the sake of brevity we denote by 
$U(x,y)=U_{s}(x,y)$
the Caffarelli--Silvestre extension of $u$. 

We choose $r_0>0$ such that the ball $\mathbb{B}_{r_0}$ contains the compact set $\mathcal{K}$, and the truncated cylinders $\mathcal{Q}_j\setminus \mathbb{B}_{r_0}$ can be covered by disjoint conical domains ${\cal C}_j$, $j=1,\dots,N$, with the common vertex at the origin. Without loss of generality we assume that $R>2(r_0+2)$.

Let $\rho_1$ and $\rho_2$ be smooth cutoff functions of $r=\sqrt{|x|^2+y^2}$ such that 
$$
\rho_1(r)=0 \quad \text{for} \quad r>r_0+2, \qquad \rho_2(r)=0 \quad \text{for} \quad r<r_0+1, \qquad \rho_1^2+\rho_2^2=1.
$$ 
Then we have
$$
|\nabla U|^2=\sum\limits_{k=1,2}\Big(|\nabla \big(U\rho_k\big)|^2-2U\nabla U\cdot \rho_k\nabla\rho_k-U^2|\nabla\rho_k|^2\Big). 
$$
Since $2 U \nabla U=\nabla (U^2)$, integration by parts gives
\begin{multline}
\label{ident}
\int\limits_0^\infty \int\limits_{\mathbb{R}^n} y^{1-2s} |\nabla U(x,y)|^2 \,dxdy=\sum\limits_{k=1,2}\Big(\int\limits_0^\infty \int\limits_{\mathbb{R}^n} y^{1-2s} |\nabla \big(U \rho_k\big)|^2 \,dxdy\\
+\int\limits_0^\infty \int\limits_{\mathbb{R}^n} U^2\rho_k \,{\rm div} (y^{1-2s}\nabla \rho_k)\,dxdy-\int\limits_{\mathbb{R}^n}U^2 y^{1-2s}\rho_k\partial_y\rho_k \,dx \Big|_{y=0}\Big)\\=:\sum_{k=1,2}(I_{k1}+I_{k2}-I_{k3}).
\end{multline}
The surface integrals $I_{k3}$ ($k=1, 2$) disappear since $\rho_k$ depend only on $r$ and $\partial_y \rho_k (x,y) = O(y)$ as $y\to +0, x\in \mathbb{R}^n$. 

To estimate terms $I_{k2}$ we split the representation \eqref{CS-ext} as follows:
$$
U(x,y)=U_1(x,y)+U_2(x,y):=
\Big(\int\limits_{\mathbb{B}_{R}}+\int\limits_{\mathbb{R}^n\setminus \mathbb{B}_{R}}\!\Big) {\cal P}_s(x-\widetilde{x},y)u(\widetilde{x})\,d\widetilde{x}.
$$ 
and note that
\begin{multline*}
|\rho_k \,{\rm div} (y^{1-2s}\nabla \rho_k)|=|y^{1-2s}\rho_k  \Delta \rho_k + 
(1-2s) y^{-2s} \rho_k\partial_y\rho_k|\\ \le Cy^{1-2s}\chi_{[r_0+1,r_0+2]}(r), 
\end{multline*}
where $\chi_G$ stands for the characteristic function of the set $G$.
This gives
$$
|I_{k2}|\le C \int\limits_0^\infty \int\limits_{\mathbb{R}^n} y^{1-2s}(U_1^2(x,y)+U_2^2(x,y))\chi_{[r_0+1,r_0+2]}(r)\,dxdy=: J_1+J_2. 
$$
The estimate of $J_1$ follows from the fact that the Poisson kernel ${\cal P}_s(\cdot,y)$ has $L_1$-norm equal to
one, see, e.g., \cite{CaSi} or \cite{MN20}. So the Young inequality yields
\begin{equation*}
J_1\le C \int\limits_0^{r_0+2}y^{1-2s}\|U_1(\cdot,y);L_2(\mathbb{R}^n)\|^2\,dy \le C \int\limits_0^{r_0+2}y^{1-2s}\|u; L_2(\mathbb{B}_{R})\|^2\,dy.
\end{equation*}
To estimate $J_2$ we notice that the inequalities $|\widetilde{x}|\geq R$ and $|x|\le r_0+2$ imply $|x-\widetilde{x}|\geq|\widetilde{x}|/2$. Using the Bunyakovsky--Cauchy--Schwarz inequality we obtain 
\begin{multline*}
J_2\le C \int\limits_{0}^{\infty} \int\limits_{\mathbb{R}^n} y^{1-2s} \chi_{[r_0+1,r_0+2]}(r) \bigg(\int\limits_{\mathbb{R}^n\setminus \mathbb{B}_{R}} |u(\widetilde{x})|\,\frac{y^{2s}}{(|\widetilde{x}|^2/4+y^2)^{\frac n2+s}}\,d\widetilde{x}\bigg)^2 \,dxdy\\
\le C(r_0) \|u;L_2(\Omega)\|^2 \int\limits_{0}^{\infty}\int\limits_{\mathbb{R}^n\setminus\mathbb{B}_{R}} y^{1-2s}\frac{y^{4s}}{(|\widetilde{x}|^2/4+y^2)^{n+2s}}\,d\widetilde{x}\, dy\\
=C(r_0) \|u;L_2(\Omega)\|^2  \int\limits_{0}^{\infty} \tau^{1-2s}\frac{\tau^{4s}}{(\tau^2+1/4)^{n+2s}} \, d\tau \int\limits_{\mathbb{R}^n\setminus\mathbb{B}_{R}} |\widetilde{x}|^{2-2s-2n}\, d\widetilde{x}\\
\le C(r_0) \|u;L_2(\Omega)\|^2 R^{2-n-2s}. 
\end{multline*}
We substitute these estimates into \eqref{ident} and arrive at
\begin{multline}
\label{ineq1}
a_s^\Omega[u]= C(s)\int\limits_0^\infty \int\limits_{\mathbb{R}^n} y^{1-2s} |\nabla U|^2 \,dxdy
\geq C(s) \int\limits_0^\infty\int\limits_{\mathbb{R}^n} y^{1-2s}|\nabla(U \rho_2)|^2\,dxdy\\
- C(r_0)R^{2-n-2s}\|u;L_2(\Omega)\|^2 -C(r_0)\|u; L_2(\mathbb{B}_{R})\|^2.
\end{multline}

Denote by $V$ the Caffarelli--Silvestre extension of the function $u\rho_2$. Since $U \rho_2$ is an admissible extension of $u\rho_2$, we have
\begin{equation*}
\int\limits_0^\infty\int\limits_{\mathbb{R}^n} y^{1-2s}|\nabla (U \rho_2) |^2\,dxdy\geq \int\limits_0^\infty\int\limits_{\mathbb{R}^n} y^{1-2s}|\nabla V |^2\,dxdy.
\end{equation*}

Now we introduce a partition of unity on the unit sphere in $\mathbb{R}^n\times \mathbb{R}$, that is a set of smooth, non-negative, zero order positively homogeneous
functions $\wp_j(x,y)\equiv\wp_j(\frac{x}{r},\frac{y}{r})$, $j=1,\dots, N$,  such that
$$
\wp_j(x,y)=\wp_j(x,-y);\qquad
\wp_j(x,0)\equiv 1\quad\mbox{for}\quad x\in{\cal C}_j;\qquad \sum\limits_{j=1}^N\wp_j^2(x,y)\equiv1.
$$

Similarly to \eqref{ident} we derive
\begin{multline}
\label{ident1}
\int\limits_0^\infty\int\limits_{\mathbb{R}^n} y^{1-2s}|\nabla V|^2\,dxdy=\sum\limits_{j=1}^N
\Big(\int\limits_0^\infty \int\limits_{\mathbb{R}^n} y^{1-2s} |\nabla \big(V \wp_j\big)|^2 \,dxdy\\
+\int\limits_0^\infty \int\limits_{\mathbb{R}^n} V^2\wp_j \,{\rm div} (y^{1-2s}\nabla \wp_j)\,dxdy-\int\limits_{\mathbb{R}^n}V^2 y^{1-2s}\wp_j\partial_y\wp_j dx \Big|_{y=0}\Big).
\end{multline}
It is easy to see that
$$
|\nabla_x\wp_j|\le \frac Cr;\qquad
|\partial_y\wp_j|\le \frac {Cy}{r^2};\qquad| \Delta\wp_j|\le \frac {C}{r^2};
$$
on the other hand, we have
$$
V(x,y)=\int\limits_{\mathbb{R}^n\setminus\mathbb{B}_{r_0+1}}\!\! {\cal P}_s(x-\widetilde{x},y)(u\rho_2)(\widetilde{x})\,d\widetilde{x},
$$
that gives $V(x,y) = O(y^{2s})$ as $y\to +0$, $x\in \mathbb{B}_{\frac{r_0}2}$.

Therefore the last term in \eqref{ident1} vanishes, and we obtain
\begin{multline*}
\int\limits_0^\infty\int\limits_{\mathbb{R}^n} y^{1-2s}|\nabla V|^2\,dxdy\geq\sum\limits_{j=1}^N
\int\limits_0^\infty \int\limits_{\mathbb{R}^n} y^{1-2s} |\nabla (V\wp_j)|^2 \,dxdy\\
-C\int\limits_0^\infty \int\limits_{\mathbb{R}^n} V^2\,\frac {y^{1-2s}}{|x|^2+y^2}\,dxdy.
\end{multline*}

Since $V \wp_j$ is an admissible extension for the function $u\rho_2\wp_j$ supported in $\mathcal{C}_j$, Corollary \ref{cor1} gives
\begin{equation*}
C(s) \int\limits_0^\infty\int\limits_{\mathbb{R}^n} y^{1-2s}|\nabla (V\wp_j) |^2\,dxdy\geq a_s^{{\cal Q}_j}[u\rho_2\wp_j]
\geq \Lambda_s \|u\rho_2\wp_j; L_2(\mathcal{C}_j)\|^2,
\end{equation*}
whereas Lemma 2.1 in \cite{MN20a} provides the estimate
\begin{equation*}
\int\limits_0^\infty \int\limits_{\mathbb{R}^n} V^2 \frac{y^{1-2s}}{|x|^2+y^2} \,dxdy \le C\||x|^{-s}u\rho_2;L_2(\Omega)\|^2.
\end{equation*}
Substituting all estimates into \eqref{ineq1}, we arrive at
\begin{multline*}
a_s^\Omega[u]\geq  \Lambda_s \sum_{j=1}^N \|u\rho_2; L_2(\mathcal{C}_j)\|^2- C\||x|^{-s}u\rho_2;L_2(\Omega)\|^2\\ - C(r_0)R^{2-n-2s}\|u;L_2(\Omega)\|^2 -C(r_0)\|u; L_2(\mathbb{B}_{R})\|^2, 
\end{multline*}
and \eqref{almost} follows.
\end{proof}

\begin{theorem}
\label{T:multitube}
Under the above assumptions, the essential spectrum of $\mathcal{A}_s^{\Omega}$ coincides with the ray $[\Lambda_s, +\infty)$, where $\Lambda_s$ is the smallest eigenvalue of $\mathcal{A}_s^\omega$.
\end{theorem}

\begin{proof}
The Weyl sequence constructed in the proof of Theorem \ref{T:tube} shows that $[\Lambda_s, +\infty)\subset\sigma_{ess}(\mathcal{A}_{s}^{\Omega})$.

To prove the opposite inclusion we need to check that if $\lambda= \Lambda_s-2\delta$ with some positive $\delta$ then $\lambda$ does not belong to the the essential spectrum of $\mathcal{A}_s^{\Omega}$. Assume the contrary and consider the corresponding Weyl sequence that is a sequence $\{u_k\}_{k=1}^{+\infty}\subset \widetilde{H}^s(\Omega)$ orthonormal in $L_2(\Omega)$ such that
\begin{equation}
\label{Th3-Weil}
 a^\Omega_s[u_k]\to \lambda \quad \mbox{as} \quad k\to+\infty.  
\end{equation}
However, choosing $R$ so large that $CR^{-2s}\le\delta$, we obtain by Theorem~\ref{T:eps} 
$$
a^\Omega_s[u_k]\geq (\Lambda_s -CR^{-2s}) - C\|u_k;L_2(
\mathbb{B}_R)\|^2\geq \lambda+\delta - C\|u_k;L_2(
\mathbb{B}_R)\|^2.
$$
By \eqref{Th3-Weil}, the sequence $\{u_k\}$ is bounded in $H^s(\mathbb{B}_R)$. By the Rellich Theorem, it is precompact in $L_2(\mathbb{B}_R)$. Since it is orthonormal, we obtain
$$
\|u_k;L_2(\mathbb{B}_R)\|\to 0\quad\Longrightarrow\quad \liminf a^\Omega_s[u_k]\geq \lambda+\delta,
$$
that contradicts \eqref{Th3-Weil}.
\end{proof}

\begin{remark}
If the cross-sections of the outlets to infinity differ, then a similar argument proves the relation $\sigma_{ess}(\mathcal{A}_s^{\Omega})=[\Lambda_s, +\infty)$, where $\Lambda_s$ is the minimal of the smallest eigenvalues for the Dirichlet fractional Laplacians on the cross-sections. 
\end{remark}

\section{Widening of the tube}\label{s:incr}

The simplest multi-tube is a locally expanded cylinder \eqref{cyl} (see Fig. \ref{fig:02}). Namely, we introduce the layer
$$
\Pi_\ell=\{x=(x',z)\subset \mathbb{R}^n \colon |z|<\ell\}, \quad \ell>0,
$$
and assume that a domain $Q'\supsetneqq Q$ coincides with $Q$ outside $\Pi_\ell$, whereas the set $Q'\cap\Pi_\ell$ is bounded. Denote for the brevity
$$
Q_\ell=Q\cap \Pi_\ell; \qquad  Q'_\ell =Q'\cap \Pi_\ell.
$$

\begin{figure}[ht]
    \centering
    \includegraphics[width=0.7\textwidth]{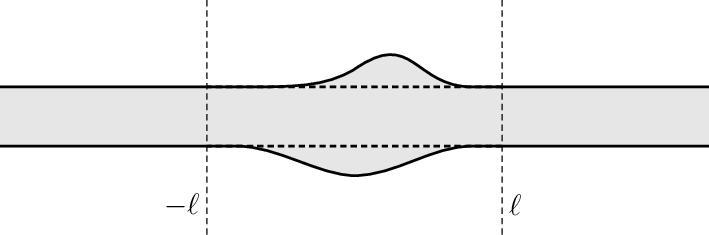}
    \caption{Locally expanded cylinder}
    \label{fig:02}
\end{figure}

By Theorem \ref{T:multitube}, we have
$$
\sigma_{ess}(\mathcal{A}_s^Q)=\sigma_{ess}(\mathcal{A}_s^{Q'})=[\Lambda_s, +\infty),
$$
where $\Lambda_s$ is the smallest eigenvalue of the operator $\mathcal{A}_s^\omega$. Denote the corresponding positive eigenfunction by $\varphi_s(x')$. 

The main result of this section is the following.

\begin{theorem}
The discrete spectrum of $\mathcal{A}_s^{Q'}$ is not empty. Namely, there is at least one eigenvalue in the interval $(0,\Lambda_s)$.
\end{theorem}

\begin{proof}
To prove this theorem we show that 
$$
\inf \sigma(\mathcal{A}_s^{Q'})< \Lambda_s.
$$
This can be done via the max-min principle by construction of a function $u\in \widetilde{H}^s(Q')$ that satisfies the inequality
$$
a_s^{Q'}[u]-\Lambda_s\|u;L_2(Q')\|^2<0.
$$
According to \eqref{qq}, it is sufficient to construct a function $W=W(x,y)$  such that
\begin{equation}
\label{A1}
C(s) \,\mathcal{E}_s^{Q'}[W]-\Lambda_s\|W(\cdot,0); L_2(Q')\|^2<0.
\end{equation}
To that end, we introduce a family of functions $W_\varepsilon$, $\varepsilon>0$, in the following way:
$$
W_\varepsilon(x,y)=U(x',y)\rho_\varepsilon(z,y)+w_\varepsilon(x,y).
$$
Here $U$ is the Caffarelli--Silvestre extension of $\varphi_s$,  the correction term $w_\varepsilon$ will be chosen later, whereas $\rho_\varepsilon$ is a cutoff function:
$$
\rho_\varepsilon(z,y):=\left\{
\begin{array}{ll}
\rho(\varepsilon|z|), &\text{if either}\ \ n\ge3\ \ \text{or}\ \  n=2\ \  \text{and}\ \ s\in(\frac 12,1);\\
\rho(\varepsilon\sqrt{y^2+z^2}),  &\text{if}\ \  n=2\ \  \text{and}\ \ s\in(0,\frac 12],
\end{array}
\right.
$$
where $\rho$ is a smooth function on $\mathbb{R}_+$, $\rho(r)\equiv 1$ for $r\le1$ and $\rho(r)\equiv0$ for $r\ge2$, and $\rho'(r)\le0$.

Inserting $W_\varepsilon$ into \eqref{A1} we obtain
\begin{equation}
\label{A2}
C(s) \,\mathcal{E}^{Q'}_s[W_\varepsilon]- \Lambda_s \|W_\varepsilon(\cdot,0); L_2(Q')\|^2={\cal I}_1+{\cal I}_2+{\cal I}_3+{\cal I}_4,
\end{equation}
where
\begin{align*}
{\cal I}_1= &\, C(s) \int\limits_0^\infty \int\limits_{\mathbb{R}^n}y^{1-2s}|\nabla U(x',y)|^2\rho_\varepsilon^2(z,y)\,dxdy 
-\Lambda_s \int\limits_Q\varphi_s^2(x')\rho_\varepsilon^2(z,0)\,dx,\\
{\cal I}_2= &\, C(s) \int\limits_0^\infty \int\limits_{\mathbb{R}^n} y^{1-2s} (2U\nabla U\cdot \rho_\varepsilon\nabla\rho_\varepsilon+U^2|\nabla\rho_\varepsilon|^2)\,dxdy,\\
{\cal I}_3= &\, 2C(s) \int\limits_0^\infty \int\limits_{\mathbb{R}^n}y^{1-2s}\nabla (U\rho_\varepsilon)\cdot\nabla w_\varepsilon \,dxdy -2\Lambda_s \int\limits_Q\varphi_s(x') \rho_\varepsilon(|z|)w_\varepsilon(x,0)\,dx,\\
{\cal I}_4= &\, C(s) \int\limits_0^\infty \int\limits_{\mathbb{R}^n}y^{1-2s}|\nabla w_\varepsilon(x,y)|^2 \,dxdy -\Lambda_s \int\limits_Q w_\varepsilon^2(x,0)\,dx.
\end{align*}
It is easy to see that
$$
{\cal I}_1\le \int\limits_{\mathbb{R}}\!\bigg(C(s) \int\limits_0^\infty \int\limits_{\mathbb{R}^{n-1}}y^{1-2s}|\nabla U(x',y)|^2\,dx'dy 
-\Lambda_s \int\limits_{\omega}\varphi_s^2(x')\,dx'\bigg)\rho^2(\varepsilon|z|)\,dz=0.
$$

Let us estimate ${\cal I}_2$. Notice that if either $n\ge3$ or $n=2$ and $s\in(\frac 12,1)$ then the first term disappears, and we have
$$
{\cal I}_2=C(s) \int\limits_0^\infty \int\limits_{\mathbb{R}^{n-1}} y^{1-2s}U^2\,dx'dy \int\limits_{\mathbb{R}}\big(\rho'_\varepsilon(|z|)\big)^2\,dz
\stackrel{(*)}{\le} C\varepsilon\int\limits_1^2\big(\rho'(r)\big)^2\,dr=O(\varepsilon)
$$
(the inequality ($*$) is due to Lemma \ref{L-2}).

The second case, $n=2$ and $s\in(0,\frac 12]$, is more tricky. We integrate by parts the first term. Similarly to \eqref{ident}, the surface integral disappears, and thus
$$
{\cal I}_2=C(s) \int\limits_0^\infty \int\limits_{\mathbb{R}^2} U^2 \big(y^{1-2s}(\partial_z\rho_\varepsilon)^2-\rho_\varepsilon\partial_y(y^{1-2s}\partial_y\rho_\varepsilon) \big)\,dxdy.
$$

As in the proof of Lemma \ref{L-2}, we use the Fourier transform in $x'$ and write
$$
\int\limits_{\mathbb{R}}U(x',y)^2\,dx'=(2\pi)^\frac 12\int\limits_{\mathbb{R}}|\mathcal{F}_{1}\varphi_s(\xi')|^2
|\widehat{p}\,(y\xi')|^2\,d\xi',
$$
where
$$
\widehat{p}\,(t)=\mathcal{F}_1\mathcal{P}_s(t,1).
$$
By the coordinate transform $y=\varepsilon^{-1} r\sin(\theta)$, $z=\varepsilon^{-1} r\cos(\theta)$ we arrive at
\begin{multline*}
|{\cal I}_2|\le C
\int\limits_1^2 \big(r(\rho^\prime(r))^2+r|\rho^{\prime\prime}(r)|+|\rho^{\prime}(r)|\big)
\int\limits_0^{\frac \pi 2}\big(\varepsilon^{-1} r\sin(\theta)\big)^{1-2s}\times\\
\times\int \limits_{\mathbb{R}}|\mathcal{F}_{1}\varphi_s(\xi')|^2 |\widehat{p}\,(\varepsilon^{-1}r\sin(\theta)\xi')|^2\,d\xi'd\theta dr\\
\le C\int \limits_{\mathbb{R}}|\mathcal{F}_{1}\varphi_s(\xi')|^2
\int\limits_1^2\int\limits_0^{\frac \pi 2}\big(\varepsilon^{-1} r\sin(\theta)\big)^{1-2s}|\widehat{p}\,(\varepsilon^{-1}r\sin(\theta)\xi')|^2\,d\theta drd\xi'.
\end{multline*}

We recall that $\widehat{p}$ decays exponentially and estimate the interior double integral as follows:
\begin{multline*}
\int\limits_1^2\int\limits_0^{\pi/2}\big(\varepsilon^{-1}r\sin(\theta)\big)^{1-2s}|\widehat{p}\,(\varepsilon^{-1}r\sin(\theta)\xi')|^2\,d\theta dr \\
\le C\int\limits_1^2\int\limits_0^{\pi/2}\big(\varepsilon^{-1}r\sin(\theta)\big)^{1-2s}\exp (-c\varepsilon^{-1}r\sin(\theta)|\xi'|)\,d\theta dr\\
\le C\int\limits_0^{\pi/2}\varepsilon^{2s-1}\exp (-c \varepsilon^{-1} \theta|\xi'|)\,d\theta
=C\varepsilon^{2s} \,\frac{1-\exp(-c \varepsilon^{-1}|\xi'|)}{|\xi'|}.
\end{multline*}
Thus,
\begin{multline*}
|{\cal I}_2|\le 
C\varepsilon^{2s} 
\bigg(\int \limits_{|\xi'|\le\varepsilon}+\int \limits_{\varepsilon\le|\xi'|\le 1}+\int \limits_{|\xi'|\ge 1}\bigg)|\mathcal{F}_{1}\varphi_{s}(\xi')|^2  \,\frac{1-\exp(-c \varepsilon^{-1}|\xi'|)}{|\xi'|}\,d\xi'\\
=:{\cal I}_{21}+{\cal I}_{22}+{\cal I}_{23}.
\end{multline*}
Now we recall that $\mathcal{F}_{1}\varphi_{s}$ is smooth, and therefore,
$$
\aligned
{\cal I}_{21}\le &\ C\varepsilon^{2s} \int\limits_0^1 \frac{1-\exp(-c t)}{t}\,dt \le C\varepsilon^{2s}; \\
{\cal I}_{22}\le &\ C\varepsilon^{2s} \int\limits_\varepsilon^1 \frac{d\xi'}{\xi'} \le C\varepsilon^{2s}\log(\varepsilon^{-1}); \\
{\cal I}_{21}\le &\ C\varepsilon^{2s} \int\limits_{\mathbb{R}} |\mathcal{F}_{1}\varphi_{s}(\xi')|^2\,d\xi' \le C\varepsilon^{2s}.
\endaligned
$$
Summing up, we obtain that in any case
$$
{\cal I}_2=O(\delta), \quad \mbox{where}\ \ \delta=\max\{\varepsilon, \varepsilon^{2s}\log(\varepsilon^{-1})\}.
$$

Now we choose $w_\varepsilon(x,y)=\delta^{\frac 12} w(x,y)$, where $w$ is a smooth function supported in $(Q_\ell'\setminus \overline{Q}_\ell)\times[0,1)$.
Then, easily, the last term in ${\cal I}_3$ vanishes, and ${\cal I}_4=O(\delta)$. Further, if $\varepsilon$ is small enough then we can drop $\rho_\varepsilon$ in ${\cal I}_3$ and recall that $U$ satisfies the equation
\begin{equation}
\label{eq:I3}
\,{\rm div} (y^{1-2s}\nabla U)=0 \quad \mbox{in} \quad \mathbb{R}^n\times \mathbb{R}_+.
\end{equation}
Therefore, the integration by parts yields
$$
{\cal I}_3=-2C(s)\delta^{\frac 12} \! \int\limits_{Q_\ell'\setminus \overline{Q}_\ell} \! \lim\limits_{y \to 
0^+} \big(y^{1-2s}~\!{\partial_y} U(x',y)\big) w(x,0)\,dx.
$$

We claim that ${\cal I}_3=-C\delta^{\frac 12}<0$ provided $w\ge0$, $w(\cdot,0)\not\equiv0$. Indeed,
changing the variable $\tau=y^{2s}$, we rewrite the equation \eqref{eq:I3} as follows:
\begin{equation}
\Delta_x U(x',\tau^{\frac 
1{2s}})+4s^2\tau^{\frac{2s-1}s}\partial^2_{\tau\tau}U(x',\tau^{\frac 
1{2s}})=0\quad\mbox{in}\quad \mathbb{R}^n\times \mathbb{R}_+,
\label{eq:KH}
\end{equation}
and
\[
{\cal I}_3=-2C(s)\delta^{\frac 12} \! \int\limits_{Q_\ell'\setminus \overline{Q}_\ell} \! 2s \lim\limits_{\tau\to 0^+}\frac{U(x',\tau^{\frac 1{2s}})}{\tau}\,w(x,0)\,dx.
\]
By the strong maximum principle, $U>0$ in $\mathbb{R}^n\times \mathbb{R}_+$. Since $U(\cdot,0)=0$ in $Q_\ell'\setminus \overline{Q}_\ell$, the differential operator in \eqref{eq:KH} satisfies the assumptions of the generalized boundary point lemma \cite{KH} (see also \cite[p.~201]{ApNa}). Namely, we have
\[
\liminf\limits_{\tau\to 0^+}\frac{U(x',\tau^{\frac 1{2s}})}{\tau}>0,\qquad  x\in Q_\ell'\setminus \overline{Q}_\ell,
\]
and the claim follows. 

Finally, we substitute all obtained estimates into \eqref{A2}. This gives
$$
C(s) \,\mathcal{E}^{Q'}_s[W_\varepsilon]- \Lambda_s \|W_\varepsilon(\cdot,0); L_2(Q')\|^2\le -C\delta^{\frac 12}+O(\delta).
$$
This, in turn, gives \eqref{A1} provided $\delta$ (and therefore $\varepsilon$) is small enough, and completes the proof.
\end{proof}

\begin{remark}
Notice that 
$\lambda_1(\mathcal{A}_s^{Q'})<\big(\lambda_1(-\Delta_{Q'})\big)^s$, cf. Remark \ref{rem1}.
\end{remark}

\paragraph*{Acknowledgements.} 
The results of Section \ref{s:straight} were obtained under support of the Russian Foundation for Basic Research (RFBR) grant 20-51-12004. The results of Sections \ref{s:mult} and \ref{s:incr} were obtained under support of the Russian Science Foundation (RSF) grant 19-71-30002.

\small
\bibliography{Frac_Lap_in_Tubes}

\end{document}